\documentclass[11pt]{amsart}

\usepackage[margin=1in]{geometry}
\usepackage{microtype}
\usepackage{mathtools,amssymb}
\usepackage{booktabs}
\usepackage{enumitem}
\usepackage{graphicx}
\usepackage{tikz}
\usetikzlibrary{arrows.meta,positioning}
\usepackage{xcolor}
\usepackage[pdfusetitle,hidelinks]{hyperref}

\hypersetup{
  colorlinks=true,
  linkcolor=blue,
  citecolor=blue,
  urlcolor=blue,
}

\numberwithin{equation}{section}

\newtheorem{theorem}{Theorem}[section]
\newtheorem{proposition}[theorem]{Proposition}
\newtheorem{lemma}[theorem]{Lemma}
\newtheorem{corollary}[theorem]{Corollary}
\theoremstyle{definition}
\newtheorem{definition}[theorem]{Definition}
\theoremstyle{remark}
\newtheorem{remark}[theorem]{Remark}

\newcommand{\A}{\mathcal A}
\newcommand{\Lcal}{\mathcal L}
\newcommand{\Pal}{\operatorname{Pal}}
\newcommand{\Lext}{\operatorname{Lext}}
\newcommand{\Rext}{\operatorname{Rext}}
\newcommand{\Bext}{\operatorname{Bext}}
\newcommand{\Pext}{\operatorname{Pext}}

\title[Low factor complexity forces palindromic equality]
      {On the number of palindromic factors of low complexity words}

\author{Dong Han Kim}
\address{Department of Mathematics Education, Dongguk University--Seoul,
30 Pildong-ro 1-gil, Jung-gu, Seoul 04620, Republic of Korea}
\email{kim2010@dgu.ac.kr}

\author{Sanghoon Kwon}
\address{Department of Mathematics Education, Catholic Kwandong University,
Gangneung 25601, Republic of Korea}
\email{skwon@cku.ac.kr, shkwon1988@gmail.com}
\date{September 15, 2026}

\hypersetup{
  pdftitle={When Low Factor Complexity Forces Palindromic Equality: The Sharp Three-Halves Threshold},
  pdfauthor={Dong Han Kim and Sanghoon Kwon}
}


\begin{document}

\begin{abstract}
Factor complexity counts the distinct factors of each length in an
infinite word, while palindromic complexity counts those invariant under
reversal.  For recurrent aperiodic words with reversal-closed language,
the total number of palindromic factors in two consecutive lengths is at
most the one-step growth of factor complexity plus two.  
We prove that equality is forced whenever the factor complexity at a
given length does not exceed three halves of that length plus one.
We construct examples showing that the bound is optimal.
We also obtain an adaptive local criterion and consequences for reversal-closed quasi-Sturmian words and palindromic defect.
\end{abstract}

\maketitle
\tableofcontents

\section{Introduction}

For an infinite word $\mathbf u$, let $p_{\mathbf u}(n)$ denote the number of its
factors of length $n$, and let $\Pal_{\mathbf u}(n)$ denote the number of its
palindromic factors of length $n$.  
General background on these two
complexity functions may be found in~\cite{AllouchePalindrome,Lothaire}.
Factor complexity is a basic quantitative invariant in combinatorics on
words and formal-language theory, and Rauzy graphs encode its local
branching in a form shared with symbolic dynamics; see Rauzy's original
account~\cite{Rauzy}.
Bal\'a\v{z}i, Mas\'akov\'a, and Pelantov\'a \cite{BMP} proved for uniformly recurrent
words that, when the language is closed under reversal,
\begin{equation}\label{eq:classical-bound}
 \Pal_{\mathbf u}(n)+\Pal_{\mathbf u}(n+1)\le \Delta p_{\mathbf u}(n)+2,
 \qquad \Delta p_{\mathbf u}(n):=p_{\mathbf u}(n+1)-p_{\mathbf u}(n).
\end{equation}
Their Rauzy-graph proof does not require uniform recurrence; the
reversal-closed version used here is recorded explicitly in
\cite[Proposition~6]{BPSproof} (see also
\cite[Theorem~1.2(ii)]{BMP}).
The equality case has a well-developed global theory.  In particular,
for a reversal-closed language, equality for every $n$ is equivalent to
the condition that every complete return to a palindrome---a factor
beginning and ending with that palindrome and containing no other
occurrence of it---is itself a palindrome; see Bucci, De Luca, Glen, and
Zamboni~\cite{BDGZ}.  This is part of the theory of rich words, whose
finite factors contain the maximum possible number of distinct
palindromes~\cite{GlenEtAl}.  
Eventual equality is closely tied to finite palindromic defect, the bounded shortfall from this maximum; see
\cite{BPSfinite,BPSproof} and Section~\ref{sec:applications} below.

Our main result is the following.

\begin{theorem}
\label{thm:main}
Let $\mathbf u$ be a recurrent aperiodic infinite word over a finite
alphabet, and suppose that its language is closed under reversal.
Then, for each integer $n\ge0$ satisfying
\begin{equation}\label{eq:main-hyp}
 p_{\mathbf u}(n)\le \frac32 n+1,
\end{equation}
one has
\begin{equation}\label{eq:main-equality}
 \Pal_{\mathbf u}(n)+\Pal_{\mathbf u}(n+1)=p_{\mathbf u}(n+1)-p_{\mathbf u}(n)+2.
\end{equation}
The coefficient $3/2$ is optimal: for every $c>3/2$, there exist an integer $n\ge0$ and a binary (that is, two-letter) uniformly recurrent
aperiodic word $\mathbf v$ with reversal-closed language such that
\[
 p_{\mathbf v}(n)\le cn+1,
 \qquad
 \Pal_{\mathbf v}(n)+\Pal_{\mathbf v}(n+1)
 <\Delta p_{\mathbf v}(n)+2.
\]
\end{theorem}

In the nontrivial case $\Delta p_{\mathbf u}(n)\ge2$, 
choose the last $m<n$ for which
$\Delta p_{\mathbf u}(m)=1$.
From $m+1$ to $n$ every increment $\Delta p$ is
at least two, whereas aperiodicity gives $p_{\mathbf u}(m) \ge m+1$.  
Then bound \eqref{eq:main-hyp} implies
\begin{equation}\label{nbound}
 n \le p_{\mathbf u}(m) + m+1.
\end{equation}
The $m$-Rauzy graph, defined formally in Section~\ref{sec:preliminaries},
is forced to be the union of two cycles $U$ and $V$ whose only common
vertex is $W$.
We set $|U| = m +k $ and $|V| = m + \ell$.  Then
\[
 k+\ell= p_{\mathbf u}(m) +1.
\]
Thus inequality in~\eqref{eq:main-hyp} places $n$ at or before
$m+k+\ell$, the first length at which a return factor from $W$ to itself can
traverse both cycle types.

The proof uses bilateral order, defined in
Section~\ref{sec:preliminaries},  which
gives the exact variation of the local gap
\[
 T_{\mathbf u}(n):=\Delta p_{\mathbf u}(n)+2-\Pal_{\mathbf u}(n)-\Pal_{\mathbf u}(n+1).
\]
At each length strictly below the endpoint in~\eqref{nbound},
every bispecial return path repeats only one of the two cycles
rather than using both $U$ and $V$.
The bilateral-order contribution of such a factor is exactly cancelled
by its palindromic extensions.  Consequently the successive
variations of $T_{\mathbf u}$ vanish up to and including that endpoint.

The same argument retains more local information than is needed for the constant $3/2$.  
If $q=p_{\mathbf u}(m) \le 2m +1$, the natural local coefficient is
\begin{equation}\label{eq:gamma-intro}
 C_{m,q}:=\frac{3q}{m+1+q}.
\end{equation}
Since $q\ge m+1$, one always has $C_{m,q}\ge3/2$.

The standard inequality~\eqref{eq:classical-bound}, its graph-theoretic
equality criterion, and the global theory of rich and finite-defect words
are prior results~\cite{BMP,BDGZ,BPSfinite,BPSproof,GlenEtAl,Rukavicka}.
Classical background on linear factor complexity and special factors is
given in~\cite{Ca96,Lothaire}.  
Our new assertion
is the one-scale implication~\eqref{eq:main-hyp}--\eqref{eq:main-equality},
its adaptive local form, and the sharpness construction of Section~\ref{sec:sharpness}.
The constant $3/2$ also appears in the different setting of minimal
subshifts generated by uniformly recurrent words, under an asymptotic
complexity-slope hypothesis~\cite{CreutzPavlov}; that work concerns
measure-theoretic structure rather than
palindromic complexity.
Shallit's recent asymptotic comparison between $\Pal_{\mathbf u}(n)$ and $p_{\mathbf u}(n)$~\cite{Shallit2026}
likewise addresses a different regime and does not impose reversal
closure.

Our quasi-Sturmian corollary $\Pal_{\mathbf u}(n)+\Pal_{\mathbf u}(n+1)=3$ and the resulting
finite-defect conclusion should be distinguished from results for
particular morphic or rotation-coding subclasses
\cite{KakoreTapsoba,MassBrlekLabbeVuillon,Starosta}.
The reflection-complexity recurrence obtained in Section~\ref{sec:applications}
is not claimed as a new theorem: it is a special case of the stronger
eventual characterization proved by Dvo\v{r}\'akov\'a and
Pelantov\'a~\cite{DvorakovaPelantova}.

The paper is organized as follows.  
Section~\ref{sec:preliminaries} collects basic properties and derives the $U,V$ two-return structure at the last unit increment of $p_{\mathbf u}(m)$.  
Section~\ref{sec:propagation} develops the return itinerary coding and proves the main theorem.
Section~\ref{sec:sharpness} exhibits the first mixed-return obstruction, first as a concrete Rauzy graph and then
as a stretched Thue--Morse family.  
Finally,
Section~\ref{sec:applications} gives the quasi-Sturmian,
finite-defect, and reflection-complexity consequences.

\section{Basic properties and auxiliary lemmas}
\label{sec:preliminaries}

Throughout Sections~\ref{sec:preliminaries} and \ref{sec:propagation},
$\mathbf u$ is a recurrent aperiodic infinite word over a finite alphabet
$\A$, and its language is closed under reversal.  
A factor is a contiguous finite block; an initial or terminal factor of a finite word is called a prefix or suffix, respectively.  
The empty word is denoted by $\varepsilon$. 
For any alphabet $\mathcal A$, we write $\mathcal A^*$ for the set of
all finite words over $\mathcal A$, including the empty word $\varepsilon$.
For a finite word $w=w_1\cdots w_r$, its length is $|w|=r$ and its
reversal is $\overline{w}=w_r\cdots w_1$.  
The set of finite factors of an infinite word $\mathbf u$ is denoted by $\Lcal(\mathbf u)$, and
$\Lcal_n(\mathbf u)$ is the subset of factors of length $n$; in
particular, $\Lcal_0(\mathbf u)=\{\varepsilon\}$.  Thus
\[
 p_{\mathbf u}(n):=\#\Lcal_n(\mathbf u),
 \quad
 \Pal_{\mathbf u}(n):=\#\{w\in\Lcal_n(\mathbf{u}):w=\overline w\},
 \quad
 \Delta p_{\mathbf u}(n):=p_{\mathbf u}(n+1)-p_{\mathbf u}(n),
\]
and $p_{\mathbf u}(0)=\Pal_{\mathbf u}(0)=1$.  
We suppress $\mathbf u$ from the notation whenever no ambiguity is possible. 

The word $\mathbf u$ is \emph{recurrent} if every factor occurs
infinitely often, and \emph{uniformly recurrent} if the gaps between
successive occurrences of each factor are bounded.  
It is \emph{eventually periodic} if it has the form $xyyy\cdots$ for finite
words $x$ and nonempty $y$, and \emph{aperiodic} otherwise.  
Its language is \emph{closed under reversal} if $w\in\Lcal$ implies
$\overline w\in\Lcal$.  
A factor $w$ is a palindrome if $w=\overline w$.

For $w\in\Lcal (\mathbf u)$, define \emph{extension sets} as 
\begin{align*}
 \Lext_{\mathbf u}(w)&:=\{a\in\A:aw\in\Lcal\},&
 \Rext_{\mathbf u}(w)&:=\{b\in\A:wb\in\Lcal\},\\
 \Bext_{\mathbf u}(w)&:=\{(a,b)\in\A^2:awb\in\Lcal\},&
 \Pext_{\mathbf u}(w)&:=\{a\in\A:awa\in\Lcal\}.
\end{align*}
A factor is left special, right special, or bispecial according as the
corresponding extension sets have cardinalities at least two.  A factor
is called \emph{special} if it is left special or right special.
The \emph{bilateral order} of $w$ is
\begin{equation}\label{eq:bilateral-order}
 b_{\mathbf u}(w):=\#\Bext_{\mathbf u}(w)-\#\Lext_{\mathbf u}(w)-\#\Rext_{\mathbf u}(w)+1.
\end{equation}
For the corresponding extension-graph viewpoint, see~\cite{BertheEtAl}.  

Recurrence ensures that all extension sets are nonempty.  In particular,
every nonbispecial factor has bilateral order zero: if one of its left
or right extension sets is a singleton, its bilateral extensions are in
bijection with the other extension set.  
The standard extension counts are
\begin{align}
 \Delta p_{\mathbf u}(n)
 &=\sum_{w\in\Lcal_n}\bigl(\#\Rext_{\mathbf u}(w)-1\bigr),
 \label{eq:first-difference}\\
 \Delta p_{\mathbf u}(n+1)-\Delta p_{\mathbf u}(n)
 &=\sum_{w\in\Lcal_n}b(w),
 \label{eq:second-difference}\\
 \Pal_{\mathbf u}(n+2)-\Pal_{\mathbf u}(n)
 &=\sum_{\substack{w\in\Lcal_n\\w=\overline w}}
       \bigl(\#\Pext_{\mathbf u}(w)-1\bigr).
 \label{eq:pal-difference}
\end{align}
See \cite{Ca97} for the first and second equalities.
The last identity follows because deleting the first and last letters from a palindrome of length $n+2$ produces a unique palindrome of length $n$.

\begin{definition}\label{def:local-gap}
The \emph{local palindromic gap at length $n$} is
\begin{equation}\label{eq:T-def}
 T_{\mathbf u}(n):=\Delta p_{\mathbf u}(n)+2-\Pal_{\mathbf u}(n)-\Pal_{\mathbf u}(n+1).
\end{equation}
\end{definition}

This local quantity should not be confused with the palindromic defect
$D(\mathbf u)$ recalled in Section~\ref{sec:applications}. 
The former measures failure of equality at one length; the latter is a global
property of the language.

For recurrent words with reversal-closed language, the classical
inequality~\eqref{eq:classical-bound} says precisely that $T_{\mathbf u}(n)\ge0$.
Moreover, $T_{\mathbf u}(n)$ is even: reversal partitions the nonpalindromic factors
of each length into two-element orbits, so
$p_{\mathbf u}(n)\equiv\Pal_{\mathbf u}(n)\pmod 2$, and similarly at length $n+1$.
Subtracting~\eqref{eq:T-def} at consecutive orders and using
\eqref{eq:second-difference}--\eqref{eq:pal-difference} gives the local variation formula
\begin{equation}\label{eq:T-variation}
 T_{\mathbf u}(n+1)-T_{\mathbf u}(n)
 =\sum_{w\in\Lcal_n(\mathbf u)}b_{\mathbf u}(w)
  -\sum_{\substack{w\in\Lcal_n (\mathbf u)\\w=\overline w}}
       \bigl(\#\Pext_{\mathbf u}(w)-1\bigr).
\end{equation}
This identity is the point at which bilateral order enters the proof.

\begin{lemma}
\label{lem:unary}
Let $\mathbf z$ be a recurrent word over $\{u,v\}$ with
reversal-closed language.  
For every $r\ge0$ such that $u^r$ is a factor of $\mathbf z$,
\begin{equation}\label{eq:unary-identity}
 b_{\mathbf z}(u^r)=\#\Pext_{\mathbf z}(u^r)-1.
\end{equation}
By symmetry, the same identity holds for $v^r$. 
\end{lemma}

\begin{proof}
If only one of the letters $u,v$ occurs in $\mathbf z$, the identity
is immediate.  Otherwise, recurrence of $\mathbf z$ implies
that both $u^rv$ and $vu^r$ occur.
If $u^r$ is not bispecial, then $u^{r+1}$ does not occur: indeed,
a finite maximal run of at least $r+1$ consecutive copies of $u$ would
give both left and both right extensions of $u^r$.
Therefore, 
$$
\Lext_{\mathbf z}(u^r) = \Rext_{\mathbf z}(u^r) = \Pext_{\mathbf z}(u^r) = \{ v \}, \quad \Bext_{\mathbf z}(u^r) = \{ (v,v)\}$$ and 
$$b_{\mathbf z}(u^r)= \#\Bext_{\mathbf z}(u^r) - \#\Lext_{\mathbf z}(u^r) - \#\Rext_{\mathbf z}(u^r) +1 = 0 = \#\Pext_{\mathbf z}(u^r)-1.
$$

Suppose that $u^r$ is bispecial.
Then $u^{r+1}$ occurs and so do $u^{r+1}v$, $vu^{r+1}$.  
$$ \Lext_{\mathbf z}(u^r) = \Rext_{\mathbf z}(u^r) = \{ u, v \}$$
and
$$
\Bext_{\mathbf z}(u^r) = \{  (x, x) : x \in \Pext_{\mathbf z}(u^r)  \}  \cup \{ (u,v), (v,u) \}.
$$
Therefore, we have 
\begin{align*}
b_{\mathbf z}(u^r) &= \#\Bext_{\mathbf z}(u^r) - \#\Lext_{\mathbf z}(u^r) - \#\Rext_{\mathbf z}(u^r) +1 \\
&= 2 + \#\Pext_{\mathbf z}(u^r) -2 -2 +1 
= \#\Pext_{\mathbf z}(u^r)-1. \qedhere
\end{align*}
\end{proof}

We shall also use the $n$-Rauzy graph $\Gamma_n(\mathbf u)$~\cite{Rauzy}.  
Vertices of $\Gamma_n(\mathbf u)$ are the factors in $\Lcal_n(\mathbf u)$. 
There is an edge from $u$ to $v$ in $\Gamma_n$ if and only if
\[
ux=yv\in\mathcal L_{n+1}(\mathbf u)
\]
for some letters $x$ and $y$.
In this case, we write
$$[uv] = ux = yv.$$
A finite sequence of edges $[u_1u_2]$, $[u_2u_3]$, \dots, $[u_{k-1}u_k]$ is called a path and we write
it as $[u_1u_2 \cdots u_{k-1}u_k]$. 
A factor of length $n + \ell$ induces a path of length $\ell$ in $\Gamma_n$. 
Moreover, a path of length $\ell$ in $\Gamma_{n+k}$ induces a path of length $\ell+k$ in $\Gamma_n$. 
However, the converse does not hold in general.
An \emph{$n$-simple path} is a path in $\Gamma_n$  whose length-$n$ prefix and suffix are special and which contains no other special factor of length $n$.  
Reversal sends an $n$-simple path from $v$ to $w$ to an $n$-simple path from $\overline w$ to
$\overline v$.
Recurrence makes $\Gamma_n$ strongly connected.  
We also define the reduced Rauzy graph $\Gamma'_n(\mathbf u)$ as follows.
Its vertices are the reversal classes $[w]:=\{w,\overline w\}$, where
$w\in\mathcal L_n(\mathbf u)$ is right or left special.  Two vertices
$[w]$ and $[v]$ are joined by an edge class
$[e]:=\{e,\overline e\}$ if $e$ or $\overline e$ is an $n$-simple path
starting at $w$ or $\overline w$ and ending at $v$ or $\overline v$.

We record the standard local equality criterion in the form needed below;
see~\cite[Proposition~6 and Lemma~7]{BPSproof}.

\begin{proposition}[Local tree criterion]
\label{prop:tree-criterion}
Fix $n$ and consider the reduced Rauzy graph $\Gamma'_n(\mathbf u)$. 
Then $T_{\mathbf u}(n)=0$
if and only if
\begin{enumerate}[label=\textup{(\roman*)}]
 \item The graph obtained from the reduced Rauzy graph $\Gamma'_n(\mathbf u)$ by removing loops is a tree; and
 \item 
Any $n$-simple path forming a loop in the reduced Rauzy graph $\Gamma'_n(\mathbf u)$ is a palindrome. 
\end{enumerate}
\end{proposition}

The following elementary overlap observation will be used twice.

\begin{lemma}
\label{lem:overlap}
Let $W$ be a word of length $m \ge 0$.
Let $U$ and $V$ be words of length $n$ with $m < n \le 2m+1$.  
If both $U$ and $V$ have prefix $W$ and suffix $\overline W$ and $U = \overline V$,
then $U = V$. 
\end{lemma}

\begin{proof}
If $m < n \le 2m$, then all letters of $U$, $V$ are determined by the prefix $W$ and suffix $\overline W$, thus $U=V$.
If $n = 2m+1$, then $U = Wa \overline W$, $V = Wb \overline W$ for some letters $a$, $b$.
Since $U = \overline V$, we have $a = b$, thus $U=V$.
\end{proof}

\begin{lemma}
\label{lem:unit}
Suppose that $\Delta p_{\mathbf u}(m)=1$.
Then there exists a palindrome $W$, which is either an $m$-simple path or a bispecial factor of length $m$,
and there are two distinct $m$-simple paths $U$, $V$. 
Taken together, $U$, $V$, and $W$ contain every nonspecial vertex
of the Rauzy graph $\Gamma_m(\mathbf u)$ exactly once.
Moreover, if we assume further that $p_{\mathbf u}(m)\le 2m +2$,  then $U$, $V$ are palindromes and 
$T_{\mathbf u}(m)=0$.
\end{lemma}

\begin{proof}
If $p_{\mathbf u}(m+1) = p_{\mathbf u}(m)+1$,
then by~\eqref{eq:first-difference}, there exists a unique right-special word of length $m$ and a unique left-special word of length $m$.
We denote by $w^*$ and ${}^*w$ the right-special and the left-special words of length $m$, respectively.
Reversal gives a unique left-special factor ${}^*w =\overline{w^*}$, with two left extensions.
We note that every vertex other than $w^*$ has one outgoing edge, and every vertex other than ${}^*w$ has one incoming edge.  

Assume first that $w^*\ne {}^*w$. 
Since $\mathbf u$ is recurrent, there exists an $m$-simple path $W$ from ${}^*w$ to $w^*$ in the Rauzy graph $\Gamma_m$. 
We write $W = [{}^*w w_1 w_2 \cdots w_{r-1} w^*]$ for $r \ge 1$.
By the uniqueness of the path from ${}^*w$ to $w^*$, we have $W = \overline W$.
Since $w^*$ has two out-going edges and ${}^*w$ has two incoming edges, we have two distinct $m$-simple paths
$U= [ w^* u_1 u_2 \dots u_{k-1} {}^*w]$ and $V = [ w^* v_1 v_2 \dots v_{\ell-1} {}^*w]$.
If $w^* = {}^*w (=: w)$, then we put  
$W = w$ and let
$U= [ w u_1 u_2 \dots u_{k-1} w]$, $V = [ w v_1 v_2 \dots v_{\ell-1} w]$ be the two distinct $m$-simple cycles.
We note that $u_1, \dots, u_{k-1}, v_1, \dots v_{\ell-1}, w_1, \dots, w_{r-1}$ are all distinct. Thus, $p_{\mathbf u} (m) = k+ \ell +r-1$ for $w^*\ne {}^*w$ and $p_{\mathbf u} (m) = k+ \ell -1$ for $w^*={}^*w$.
Therefore, we have
\begin{equation}\label{eq2.7}
|U| + |V| = m +k + m +\ell \le 2m + p_{\mathbf u}(m) + 1 \le 4m + 3. 
\end{equation}

We note that $\overline U$ and $\overline V$ are also paths from $\overline{{}^*w} =w^*$ to $\overline{w^*} = {}^*w$.
Since there are only two outgoing edges from $w^*$ and two incoming edges to ${}^*w$, we have 
either $U = \overline U$, $V = \overline V$ or $U = \overline V$. 
Suppose $U = \overline V$. 
Then $|U| = |V|$ and \eqref{eq2.7} implies that
$$|U| = |V| \le 2m + 1.$$ 
By Lemma~\ref{lem:overlap}, we deduce that $U = V$,
which is a contradiction.
Therefore, $U = \overline U$, $V = \overline V$.
Proposition~\ref{prop:tree-criterion} yields $T_{\mathbf u}(m)=0$.
\end{proof}

We next describe what happens when an increment of one in $p_{\mathbf u}(m)$ is followed by a larger increment.  

\begin{lemma}
\label{lem:strong-transition}
Let $m\ge0$ and suppose that
\[
 \Delta p_{\mathbf u}(m)=1,
 \qquad
 \Delta p_{\mathbf u}(m+1)\ge2.
\]
Then the unique right-special and left-special factors of length $m$
coincide in a palindrome $W$.  Moreover,
\[
 \#\Lext(W)=\#\Rext(W)=2,
 \qquad
 \#\Bext(W)=4,
 \qquad
 b(W)=1,
\]
and $\Delta p_{\mathbf u}(m+1)=2$.
\end{lemma}

\begin{proof}
As in the proof of Lemma~\ref{lem:unit}, there is a unique right-special factor $w^*$ and a unique left-special factor ${}^*w = \overline{w^*}$.  
If $w^* \ne {}^*w$, no length-$m$ factor is bispecial.  
Every summand in \eqref{eq:second-difference} is then zero, contradicting $\Delta p(m+1)>\Delta p(m)$.  
Hence $w^* = {}^*w =:W$, and $W=\overline W$.
All factors other than $W$ have bilateral order zero.  
Since $W$ has two left and two right extensions,
\[
 b(W)=\#\Bext(W)-3\le1.
\]
Equation~\eqref{eq:second-difference} and $\Delta p(m+1)>\Delta p(m)$ force
$b(W)=1$, so all four bilateral extensions occur and
$\Delta p_{\mathbf u}(m+1)=2$.
\end{proof}

\section{Proof of the main theorem}
\label{sec:propagation}

In the preceding section, we showed that if $\Delta p_{\mathbf u}(m)=1$ and $\Delta p_{\mathbf u}(m+1)\ge2$, then the Rauzy graph $\Gamma_m(\mathbf u)$ is the union of two distinct
$m$-simple cycles $U$ and $V$ whose only common vertex is the
bispecial palindrome $W$.
Let $W=w$, $U = [wu_1\dots u_{k-1}w]$, $V = [w v_1 \dots v_{\ell-1} w]$ as in the proof of Lemma~\ref{lem:unit}.
Then 
\begin{equation}\label{eq:cycle-sum}
 p_{\mathbf u}(m) = k+\ell - 1, \qquad |U|=m+k, \quad |V|= m+\ell. 
\end{equation}
For $X, Y  \in \{ U,V\}$, concatenating two cycles $X  = [w x_1\dots x_{s-1}w]$ and $Y= [w y_1\dots y_{t-1} w]$ in the Rauzy graph $\Gamma_m(\mathbf u)$ gives the factor 
$$W \tilde X \tilde Y = X \tilde Y = [w x_1\dots x_{s-1} w y_1\dots y_{t-1}w],$$
where $\tilde U$ and $\tilde V$ are finite words satisfying 
$$U = W \tilde U, \quad V = W \tilde V. $$
Let $\Phi$ be the morphism from $\{ u, v\}^*$ to $\mathcal A^*$ given by $\Phi(u) = \tilde U$, $\Phi(v) = \tilde V$.
By the concatenated cycles in $\Gamma_m(\mathbf u)$, we mean a factor 
 $$
F(y) = W \Phi(y) \in \mathcal L (\mathbf u),
$$
where $y$ is a finite word over the two-letter alphabet $\{u,v\}$; it will be called a \emph{return itinerary}. 
We note that $F(\varepsilon) = W$.

\begin{lemma}
\label{lem:extension-transfer}
Suppose that the Rauzy graph $\Gamma_m (\mathbf u)$ consists of a bispecial palindromic vertex $W$ and two distinct $m$-simple palindromic cycles $U$, $V$ from $W$ to itself. 
Choose an occurrence of $W$ in $\mathbf u$.  Reading successive returns
to $W$, code the cycle $U$ by $u$ and the cycle $V$ by $v$, and let
$\mathbf z$ be the resulting infinite word.  Let $y$ be a factor of
$\mathbf z$.

Then
\begin{enumerate}[label=\textup{(\roman*)}]
 \item $\mathbf z$ is recurrent and its language is closed under reversal;
 \item for every itinerary factor $y$,
 \begin{equation}\label{eq:return-reversal}
   \overline{F(y)}=F(\overline y);
 \end{equation}
 \item the bilateral extensions of $y$ in $\mathbf z$ are in bijection with the bilateral extensions of $F(y)$ in $\mathbf u$.  
 This bijection preserves equal left and right letters.  Consequently,
 \begin{equation}\label{eq:extension-transfer}
   b_{\mathbf u}(F(y))=b_{\mathbf z}(y),
 \end{equation}
 and, whenever $y$ is a palindrome,
 \begin{equation}\label{eq:pal-extension-transfer}
   \#\Pext_{\mathbf u}(F(y))=
   \#\Pext_{\mathbf z}(y).
 \end{equation}
\end{enumerate}
If all four bilateral extensions of $W$ occur, then $uu$, $uv$, $vu$,
and $vv$ all occur in $\mathbf z$.
\end{lemma}

\begin{proof}
Since $U$ and $V$ are the only $m$-simple return cycles from $W$ to
itself, every factor of $\mathbf u$ that begins and ends with $W$ has a
unique decomposition into these cycles and hence is of the form $F(y)$
for a unique itinerary $y$.
If an itinerary factor $y$ occurs in $\mathbf z$, then $F(y)$ occurs in $\mathbf u$. Since $\mathbf u$ is recurrent, $F(y)$ occurs infinitely often in
$\mathbf u$. By uniqueness of the decomposition into the return cycles $U$ and $V$,
each such occurrence corresponds to an occurrence of $y$ in $\mathbf z$. Thus $\mathbf z$ is recurrent.

For $Y \in\{ \tilde U, \tilde V\}$, palindromicity of $W Y$ gives
\begin{equation}\label{eq:return-commutation}
 W Y=\overline{Y} W.
\end{equation}
Moving $W$ successively to the left by this identity gives, for
$y=y_1\cdots y_s$, if we write $\Phi(y_i) = Y_i \in \{ \tilde U, \tilde V\}$, $i =1 ,\dots, s$, then
$$
\overline{F(y)} = \overline{ W \Phi(y_1) \cdots \Phi(y_s) } = \overline {W Y_1 \cdots Y_s} = 
\overline{Y_s} \cdots \overline {Y_1} W = W Y_s \cdots Y_1 = F(\overline y).
$$
which proves~\eqref{eq:return-reversal}.  
Since the language of $\mathbf u$ is closed under reversal, $\overline{F(y)}=F(\overline y)$ is a factor of $\mathbf u$; 
uniqueness of the itinerary then shows that $\overline y$ belongs to the language of $\mathbf z$.  
This proves~\textup{(i)} and
\textup{(ii)}.

It remains to verify the extension statement. 
For $X\in\{U,V\}$, let $\alpha_X$ be the letter immediately preceding the suffix $W$ in $X$.
Since $U$, $V$ are palindromes, $\alpha_X$ is the letter immediately following the prefix $W$ of $X$.
Since the two cycles leave $W$ along distinct outgoing edges, one has $\alpha_U \ne \alpha_V$.
Note that $\alpha_X$ is the first and last letter of $X$ if $W=\varepsilon$.
We check
$$\Lext(W) = \Rext(W) = \{ \alpha_U, \alpha_V\}.$$

For letters $a,b\in\{u,v\}$, an occurrence of $ayb$ in $\mathbf z$
corresponds exactly to an occurrence of $F(y)$ preceded by
$\alpha_{F(a)}$ and followed by $\alpha_{F(b)}$.
Uniqueness of the itinerary proves that
\[
 (a,b)\longmapsto(\alpha_{F(a)},\alpha_{F(b)})
\]
is a bijection from $\Bext_{\mathbf z}(y)$ to
$\Bext_{\mathbf u}(F(y))$; its restrictions give bijections of the left
and right extension sets as well.
The bilateral-order identity~\eqref{eq:extension-transfer} follows from
the definition, while preservation of equal boundary letters gives
\eqref{eq:pal-extension-transfer}.  By~\eqref{eq:return-reversal},
$F(y)$ is palindromic whenever $y$ is.

Finally, the same bijection with $y=\varepsilon$ identifies the four
bilateral extensions of $W$ with the four length-two itinerary factors.
\end{proof}

\begin{proposition}
\label{prop:propagation}
Suppose that the Rauzy graph $\Gamma_m (\mathbf u)$ consists of a bispecial palindromic vertex $W$ and two distinct $m$-simple palindromic cycles $U$, $V$ from $W$ to itself. 
We assume that all four bilateral extensions of $W$ occur and let $|U|= k+m$, $|V|= m+\ell$.  
Then
\begin{equation}\label{eq:propagation-window}
 T_{\mathbf u}(n)=0
 \qquad\text{for every }m\le n \le m+k+\ell.
\end{equation}
\end{proposition}

\begin{proof}
By Proposition~\ref{prop:tree-criterion}, we have $T_{\mathbf u}(m)=0$.

Record successive cycles to $W$ as the letters $u$ and $v$, and let $\mathbf z$ be the resulting return itinerary.
Lemma~\ref{lem:extension-transfer} shows that $\mathbf z$ is recurrent and
reversal closed, transfers bilateral and palindromic extensions, and
shows that $uu$, $uv$, $vu$, and $vv$ all occur.

Fix $d$ with $0\le d<k+\ell$, and let $w$ be any bispecial factor of length $m+d$.  
Since $W$ is the unique left-special factor of length $m$, the length-$m$ prefix of $w$ must be $W$.  
Similarly, uniqueness of the right-special factor forces the length-$m$ suffix of $w$ to be $W$.  
Thus $w=F(y)=W\Phi(y)$ for an itinerary word $y$ of $|F(y)| = |W \Phi(y)| = m+d$, where $|\Phi(u)| = |\tilde U|= k$ and $|\Phi(v)| = |\tilde V| = \ell$.  
Since $d<k+\ell$, the word $y$ cannot contain both letters.  
Thus $y=u^r$, $y=v^s$, or $y=\varepsilon$ when $d=0$.

The itinerary $y$ is a palindrome, so
Lemma~\ref{lem:extension-transfer} implies that $w$ is a palindrome and
that
\[
 b_{\mathbf u}(w)=b_{\mathbf z}(y),
 \qquad
 \#\Pext_{\mathbf u}(w)=\#\Pext_{\mathbf z}(y).
\]
Lemma~\ref{lem:unary} shows that the contribution of every bispecial
factor to the right-hand side of~\eqref{eq:T-variation} is zero.  
A nonbispecial factor has bilateral order zero.  
If it is a palindrome, reversal symmetry gives the same unique letter on the left and on the
right, and recurrence then gives exactly one palindromic extension; it again contributes zero.
Therefore
\[
 T_{\mathbf u}(m+d+1)-T_{\mathbf u}(m+d)=0
 \qquad(0\le d<k+\ell).
\]
Starting from $T_{\mathbf u}(m)=0$ proves~\eqref{eq:propagation-window}.
\end{proof}

\begin{remark}
The strict inequality $d<k+\ell$ is the source of the bound $n\le m+k+\ell$.
A mixed itinerary such as $uv$ or $vu$ may first represent a factor
of length $m+k+\ell$.  Such a factor can affect
$T_{\mathbf u}(m+k+\ell+1)-T_{\mathbf u}(m+k+\ell)$, but not the already
established equality $T_{\mathbf u}(m+k+\ell)=0$.
The sharpness examples in Section~\ref{sec:sharpness} exploit exactly this first mixed-return obstruction.
\end{remark}

We first retain the actual value
$q=p_{\mathbf u}(m)$ and obtain an adaptive coefficient.

\begin{theorem}[Adaptive local coefficient]\label{thm:adaptive}
Let $n>m\ge 0$, suppose that
\[
 \Delta p_{\mathbf u}(m)=1,
 \qquad
 \Delta p_{\mathbf u}(j)\ge2\quad(m<j\le n).
\] 
Assume $p_{\mathbf u}(m) \le2m+1$ and let
$$
q:=p_{\mathbf u}(m) \ge m+1,  \qquad C_{m,q} :=\frac{3q}{m+1+q} \ge \frac 32.
$$
If
\begin{equation}\label{eq:adaptive-hyp}
 p_{\mathbf u}(n)\le C_{m,q} \, n+1,
\end{equation}
then $T_{\mathbf u}(n)=0.$
\end{theorem}

\begin{proof}
Using $\Delta p_{\mathbf u}(m)=1$ and the subsequent lower bounds gives
\begin{equation}\label{eq:growth-lower}
 p_{\mathbf u}(n)
 =p_{\mathbf u}(m)+\Delta p_{\mathbf u}(m)+\sum_{j=m+1}^{n-1}\Delta p_{\mathbf u}(j)
 \ge q + 1 +2(n-m-1) = q + 2n - 2m -1.
\end{equation}
Together with~\eqref{eq:adaptive-hyp}, this yields
\begin{equation*}
 \bigl(2- C_{m,q} \bigr)n
 = \frac{2m+2-q}{m+1+q}n 
 \le 2m+2-q.
\end{equation*}
Since $q\le2m+1$, 
we have 
\begin{equation}\label{eq:n-window}
 n\le m+1+q.
\end{equation}

Lemma~\ref{lem:strong-transition}, together with
Lemma~\ref{lem:unit}, supplies the two palindromic $m$-simple cycles
$U,V$ with $|U|=m+k$, $|V|=m+\ell$, and $k+\ell=q+1$.
Hence $n\le m+k+\ell$, and Proposition~\ref{prop:propagation} gives $T_{\mathbf u}(n)=0$.
\end{proof}

\begin{proof}[Proof of 
Theorem~\ref{thm:main}]
If $n = 0$, then 
$$
\Pal_{\mathbf u}(0) = p_{\mathbf u}(0) = 1, \quad 
\Pal_{\mathbf u}(1) = p_{\mathbf u}(1).
$$
Therefore, we have $T_{\mathbf u}(0) = 0$.

Now, we assume~\eqref{eq:main-hyp} and $n \ge 1$.  
If $\Delta p(n)=1$, then 
$p_{\mathbf u}(n)\le\frac32n+1\le2n+2$, so Lemma~\ref{lem:unit} gives $T_{\mathbf u}(n)=0$.

Suppose that $\Delta p(n)\ge2$.  
There exists $m<n$ with
$\Delta p(m)=1$.  
Indeed, a standard consequence of the Morse--Hedlund theorem~\cite{MorseHedlund} is that an aperiodic word satisfies
\[
 p(j+1)>p(j)
 \qquad\text{for every }j\ge0;
\]
equality at one length would force eventual periodicity.  
Thus
$\Delta p(j)>0$ for every $j\ge0$.  If no preceding increment were one,
then every increment from $0$ through $n-1$ would be at least two,
giving $p(n)\ge2n+1$, contrary to~\eqref{eq:main-hyp}.  
Choose the last such $m$, and put $q=p(m)$.
Then $\Delta p(j)\ge2$ for $m<j\le n$.
Thus, \eqref{eq:main-hyp} and \eqref{eq:growth-lower}  give
\begin{equation}\label{eq:three-halves-numerics}
q + \frac n2 \le 2m + 2.
\end{equation}
For $n\ge1$, equation~\eqref{eq:three-halves-numerics} and the
integrality of $q$ give $q\le2m+1$.
Moreover, \eqref{eq:main-hyp} implies
$$
p_{\mathbf u}(n) \le \frac 32 n+1 \le C_{m,q} \, n+1.
$$
Hence, by Theorem~\ref{thm:adaptive}, we deduce 
$T_{\mathbf u}(n) = 0. $
\end{proof}

The coefficient $3/2$ is the largest universal coefficient that always places $n$ before the first
possible mixed $U,V$ itinerary.

\section{Sharpness at the first mixed return}
\label{sec:sharpness}

The proof of Theorem~\ref{thm:main} reaches the first length at
which a bispecial path can use both $U$ and $V$.  This is not an artifact
of the method.  At the following variation step, the first
mixed-return scale can create a cycle in the reduced Rauzy graph and a
positive local gap.
We first display the smallest member
of the sharpness family, and then perform the general enumeration.

Let $\mathbf t=0110100110010110\cdots$ be the Thue--Morse word, the
unique infinite fixed point beginning in $0$ of the morphism
\[
 \tau(0)=01,
 \qquad
 \tau(1)=10.
\]
For an integer $M\ge1$, define a morphism $h_M$, extended from letters to
words by concatenation, by
\begin{equation}\label{eq:hM}
 h_M(0)=a,
 \qquad
 h_M(1)=a^M b,
 \qquad
 \mathbf u_M=h_M(\mathbf t),
\end{equation}
and put
\begin{equation}\label{eq:NM}
 N_M:=2M+3.
\end{equation}

\subsection{The seed example and its Rauzy graph}

Take $M=1$ in~\eqref{eq:hM}, so that
$h_1(0)=a$, $h_1(1)=ab$, and consider length $N_1=5$.  Direct factor
enumeration gives
\[
 p_{\mathbf u_1}(5)=9,
 \qquad
 p_{\mathbf u_1}(6)=13,
 \qquad
 \Delta p_{\mathbf u_1}(5)=4.
\]
The palindromes of length five are
\[
 aabaa,\qquad ababa,\qquad baaab,
\]
and the only palindrome of length six is $abaaba$.  Hence
\[
 T_{\mathbf u_1}(5)=4+2-3-1=2.
\]

This example also exposes the $U,V$ mechanism without any asymptotics.
The last $m<5$ with $\Delta p(m)=1$ is $m=1$, the common vertex is
$W=a$, and the two appended return blocks may be chosen as
\[
 \widetilde U=a,\qquad \widetilde V=ba.
\]
Thus $U=W\widetilde U=aa$ and $V=W\widetilde V=aba$ are
palindromes, while $k=|\widetilde U|=1$ and
$\ell=|\widetilde V|=2$.  Proposition~\ref{prop:propagation} guarantees
$T_{\mathbf u_1}(N)=0$ only through $N=m+k+\ell=4$.  
At relative length
$k+\ell=3$, the first mixed itineraries give
\[
 W\widetilde U\widetilde V=aaba,
 \qquad
 W\widetilde V\widetilde U=abaa.
\]
These are reversal-paired bispecial factors, each of bilateral order one.  
Every other length-$4$ factor has bilateral order zero, and
the sole length-$4$ palindrome $baab$ has exactly one palindromic extension.  
Hence the variation formula shows that the two mixed factors
are the only uncancelled contributions and gives
$T_{\mathbf u_1}(5)-T_{\mathbf u_1}(4)=2$, which agrees with the direct count above.

Figure~\ref{fig:seed-rauzy}(a) gives the complete length-$5$ Rauzy graph.
The shaded vertices are palindromes.  Its special vertices fall into the
four reversal classes
\[
 X=[aaaba,abaaa],\quad Y=[aabaa],\quad
 Z=[aabab,babaa],\quad Q=[abaab,baaba].
\]
After nonspecial vertices are suppressed, reversal classes are
identified, and reversal loops are deleted, the reduced Rauzy graph is the square in Figure~\ref{fig:seed-rauzy}(b).  
The fixed loops deleted at $X,Z,Q$ are palindromic; the failure of the local criterion comes from
the remaining $4$-cycle.  
Thus the graph records precisely the first mixed-return obstruction, while the direct count above gives its exact local gap $T_{\mathbf u_1}(5)=2$.

\begin{figure}[htbp]
\centering
\begin{minipage}[t]{0.65\textwidth}
\centering
\begin{tikzpicture}[
  scale=1,transform shape,
  >={Stealth[length=2.2mm]},
  vertex/.style={draw,rounded corners=1.5pt,inner sep=2.2pt,                 font=\ttfamily\scriptsize,fill=white},
  pal/.style={fill=black!13},
  every edge/.style={draw,->,line width=.42pt}
]
\node[vertex]     (A) at (1,1.5)   {aaaba};
\node[vertex,pal] (B) at (3,1.5) {aabaa};
\node[vertex]     (C) at (0,0)  {aabab};
\node[vertex]     (D) at (5,1.5)   {abaaa};
\node[vertex]     (E) at (4,0)   {abaab};
\node[vertex,pal] (F) at (3,-1.2)   {ababa};
\node[vertex,pal] (G) at (3,2.7)   {baaab};
\node[vertex]     (H) at (2,0)   {baaba};
\node[vertex]     (I) at (6,0)   {babaa};

\path (A) edge (B) edge (C)
      (B) edge (D) edge (E)
      (C) edge (F)
      (D) edge (G)
      (E) edge (H)
      (F) edge (I);
\draw[->] (G) to (A);
\draw[->] (H) to (B);
\draw[->] (H) to (C);
\draw[->] (I) to (D);
\draw[->] (I) to (E);
\end{tikzpicture} \\
\centering
{\small (a) The Rauzy graph $\Gamma_5$.}
\end{minipage}\hfill
\begin{minipage}[t]{0.30\textwidth}
\centering
\begin{tikzpicture}[
  >={Stealth[length=2.2mm]},
  qvertex/.style={draw,circle,minimum size=7mm,inner sep=1pt,
                  font=\small},
  qedge/.style={draw,line width=.75pt}
]
\node[qvertex] (X) at (0,2.4) {$X$};
\node[qvertex] (Y) at (2.4,2.4) {$Y$};
\node[qvertex] (Z) at (0,0) {$Z$};
\node[qvertex] (Q) at (2.4,0) {$Q$};
\draw[qedge] (X)--(Y)--(Q)--(Z)--(X);
\end{tikzpicture} \\
\centering
{\small (b) The reduced Rauzy graph $\Gamma'_5$.}
\end{minipage}
\caption{The first mixed-return obstruction for $\mathbf u_1$ at length $5$.  
The reduced graph contains one cycle and is not a tree.}
\label{fig:seed-rauzy}
\end{figure}

For reference, the solid arrows in Figure~\ref{fig:seed-rauzy}(a) are
the following adjacency relations:
\[
\begin{array}{rcl@{\qquad}rcl}
aaaba&\longrightarrow&aabaa,aabab & aabaa&\longrightarrow&abaaa,abaab,\\
aabab&\longrightarrow&ababa       & abaaa&\longrightarrow&baaab,\\
abaab&\longrightarrow&baaba       & ababa&\longrightarrow&babaa,\\
baaab&\longrightarrow&aaaba       & baaba&\longrightarrow&aabaa,aabab,\\
babaa&\longrightarrow&abaaa,abaab.&&&
\end{array}
\]

\subsection{The stretched Thue--Morse family}

\begin{lemma}\label{lem:family-properties}
For every $M\ge1$, the word $\mathbf u_M$ is binary, uniformly recurrent,
aperiodic, and its language is closed under reversal.
\end{lemma}

\begin{proof}
We use the standard facts that the Thue--Morse word is uniformly
recurrent and aperiodic, and that its language is closed under reversal;
see~\cite[Chapters~1 and~10]{AlloucheShallit} and~\cite{Lothaire}.
Both letter images under $h_M$ are nonempty, so $h_M$ is nonerasing;
uniform recurrence is therefore preserved and $\mathbf u_M$ is
uniformly recurrent.  Each $b$ in $\mathbf u_M$ is the last letter in
the image of a unique $1$ in $\mathbf t$.  If $r$ zeros lie between two
consecutive $1$'s, then the corresponding $b$'s are separated by
$M+1+r$ positions.  Eventual periodicity of $\mathbf u_M$ would make
this $b$-gap sequence eventually periodic.  Subtracting $M+1$ recovers
the successive zero-run lengths $r_1,r_2,\ldots$ between the $1$'s.
An eventually periodic sequence of these zero-run lengths would
reconstruct an eventually periodic tail
$1\,0^{r_1}\,1\,0^{r_2}\,1\,0^{r_3}\cdots$ of $\mathbf t$,
contradicting the aperiodicity of the Thue--Morse word.

For reversal closure, define
\[
 \widehat h_M(0)=a,
 \qquad
 \widehat h_M(1)=ba^M.
\]
For every finite word $v$,
\begin{equation}\label{eq:conjugate-morphisms}
 h_M(v)a^M=a^M\widehat h_M(v),
 \qquad
 \overline{h_M(v)}=\widehat h_M(\overline v).
\end{equation}
Passing to longer prefixes in the first identity shows that
$\widehat h_M(\mathbf t)$ is obtained from $h_M(\mathbf t)$ by
deleting its first $M$ letters.  Uniform recurrence therefore gives the
same factor language
for the two morphic images.  If $w$ is a factor of $h_M(\mathbf t)$,
choose a factor $v$ of $\mathbf t$ whose image contains $w$.  Reversal
closure of the Thue--Morse language and the second identity show that
$\overline w$ is a factor of $\widehat h_M(\mathbf t)$, hence of
$\mathbf u_M$.  This proves reversal closure.
\end{proof}

\begin{lemma}[Thue--Morse gap catalogue]\label{lem:tm-gaps}
Let $r$ denote the number of zeros between two consecutive occurrences
of $1$ in $\mathbf t$.  Then $r\in\{0,1,2\}$.  Two consecutive gap
parameters are unequal, and every ordered unequal pair occurs.
\end{lemma}

\begin{proof}
It is classical that the Thue--Morse word $\mathbf t$ is overlap-free;
see, for example,~\cite[Chapter~1]{AlloucheShallit}.
Recall that an \emph{overlap} is a finite word of the form $cxcxc$,
where $c$ is a letter and $x$ is a possibly empty word.  Thus an
infinite word is overlap-free if none of its factors has this form.

Since $000=0\varepsilon0\varepsilon0$ is an overlap, at most two zeros
can occur between consecutive occurrences of $1$ in $\mathbf t$.
Equal consecutive gap parameters would produce, for $r=0,1,2$,
respectively,
\[
 111,\qquad 10101,\qquad 1001001,
\]
each of which is an overlap.  Hence consecutive parameters are unequal.
Conversely,
the six ordered unequal pairs
\[
 (0,1),\ (1,0),\ (0,2),\ (2,0),\ (1,2),\ (2,1)
\]
are witnessed, respectively, by the Thue--Morse factors
\[
 1101,\quad 1011,\quad 11001,\quad 10011,
 \quad 101001,\quad 100101. \qedhere
\]
\end{proof}

We give the factor enumeration because the additive constant is relevant
for sharpness.

\begin{proposition}[Exact count at the critical scale]
\label{prop:sharp-count}
For $N_M=2M+3$,
\begin{align}
 p_{\mathbf u_M}(N_M)&=3M+6,
 &p_{\mathbf u_M}(N_M+1)&=3M+10,
 \label{eq:sharp-p}\\
 \Pal_{\mathbf u_M}(N_M)+\Pal_{\mathbf u_M}(N_M+1)&=4.
 \label{eq:sharp-pal}
\end{align}
Consequently
\begin{equation}\label{eq:sharp-gap}
 \Delta p_{\mathbf u_M}(N_M)=4,
 \qquad
 T_{\mathbf u_M}(N_M)=2.
\end{equation}
\end{proposition}

\begin{proof}
Between two consecutive letters $b$ in $\mathbf u_M$ there are
$M+r$ letters $a$, where $r\in\{0,1,2\}$ is the number of zeros between
the corresponding consecutive $1$'s of $\mathbf t$.  Thus the distance
between the two $b$'s is $M+1+r$.  Lemma~\ref{lem:tm-gaps} gives the
complete list of one- and two-gap possibilities needed below.

Every factor of length $N_M$ or $N_M+1$ is therefore determined by the
positions of its $b$'s.  There is no factor without $b$, because an
$a$-run has length at most $M+2$.  Nor can four $b$'s occur: three
successive $b$-distances have span at least $3(M+1)+1>2M+4$, because
consecutive gap parameters are unequal.  The following table counts the factors with one,
two, or three occurrences of $b$; the number of palindromes in each class
is shown in parentheses.  Put $\epsilon_M=1$ if $M$ is odd and
$\epsilon_M=0$ if $M$ is even.
\begin{center}
\begin{tabular}{c@{\qquad}c@{\qquad}c@{\qquad}c@{\qquad}c}
\toprule
length & one $b$ & two $b$'s & three $b$'s & total \\
\midrule
$2M+3$ & $3\ (1)$ & $3M+3\ (1+\epsilon_M)$ & $0\ (0)$
        & $3M+6\ (2+\epsilon_M)$ \\
$2M+4$ & $2\ (0)$ & $3M+6\ (2-\epsilon_M)$ & $2\ (0)$
        & $3M+10\ (2-\epsilon_M)$ \\
\bottomrule
\end{tabular}
\end{center}

For completeness, the entries are obtained as follows.  A factor with
one $b$
has the form $a^xba^y$.  Since the neighboring $a$-runs have length at
most $M+2$, the conditions $x+y=2M+2$ and $x+y=2M+3$ give respectively
three and two possibilities.  More explicitly, the pairs $(x,y)$ are
\[
 (M,M+2),(M+1,M+1),(M+2,M)
\]
at length $2M+3$, and
\[
 (M+1,M+2),(M+2,M+1)
\]
at length $2M+4$.  The six ordered unequal gap pairs in
Lemma~\ref{lem:tm-gaps} realize each of these configurations without
introducing another $b$.

A factor with two $b$'s has internal
distance $d=M+1+r$.  Its first $b$ can occupy
\[
 M+2-r \quad\text{positions at length }2M+3,
 \qquad
 M+3-r \quad\text{positions at length }2M+4.
\]
Summing over $r=0,1,2$ gives $3M+3$ and $3M+6$.  At length
$2M+3$ no third $b$ can enter, so every placement occurs around any
realization of the gap $r$.  At length $2M+4$, at most one side of the
window can admit a third $b$.  Choosing on that side a neighboring gap
parameter unequal to $r$ and large enough to place the next $b$ outside
the window is always possible by Lemma~\ref{lem:tm-gaps}: choose $2$
when $r=0$ or $1$, while either unequal value works when $r=2$.
Thus the displayed placement counts are exact.
Three $b$'s require two unequal gap parameters.  At length $2M+3$ none
fit; at length $2M+4$ exactly $(0,1)$ and $(1,0)$ fit, giving two
factors.

A one-$b$ factor is palindromic exactly when $x=y$.  For a two-$b$
factor, palindromicity likewise says that the two exterior $a$-runs have
equal length.  At length $2M+3$ this is possible for one value of $r$ if
$M$ is even and for two values if $M$ is odd; at length $2M+4$ the counts
are reversed.  The two three-$b$ factors are reversals of one another and
are not palindromes.  This proves the table and
\eqref{eq:sharp-p}--\eqref{eq:sharp-pal};
\eqref{eq:sharp-gap} follows immediately.
\end{proof}

\begin{proof}[Proof of sharpness in Theorem~\ref{thm:main}]
By~\eqref{eq:sharp-p} and~\eqref{eq:NM},
\begin{equation}\label{eq:ratio-sharp}
 \frac{p_{\mathbf u_M}(N_M)-1}{N_M}
 =\frac32+\frac{1}{4M+6}.
\end{equation}
Given $c>3/2$, choose $M$ so large that
$1/(4M+6)<c-3/2$.  Equation~\eqref{eq:ratio-sharp} gives
$p_{\mathbf u_M}(N_M)\le cN_M+1$, while
 Proposition~\ref{prop:sharp-count} gives
\[
 \Pal(N_M)+\Pal(N_M+1)=4<6=\Delta p(N_M)+2.
\]
Lemma~\ref{lem:family-properties} supplies all required dynamical and
symmetry properties.
\end{proof}

\section{\texorpdfstring{Consequences}{Consequences}}
\label{sec:applications}

\subsection{Quasi-Sturmian words and finite defect}

An aperiodic word is \emph{quasi-Sturmian} if there are constants $d$ and
$n_0$ such that
\begin{equation}\label{eq:quasi-sturmian}
 p_{\mathbf u}(n)=n+d
 \qquad(n\ge n_0).
\end{equation}

Recall that the palindromic defect of a finite word $w$ is
\[
 D(w)=|w|+1-\#\{\text{distinct palindromic factors of }w\},
\]
where the empty word is counted, and
$D(\mathbf u)=\sup\{D(w):w\in\Lcal(\mathbf u)\}$.  The
Brlek--Reutenauer formula, proved in full generality in~\cite{BPSproof},
states that, for an infinite word with reversal-closed language,
\begin{equation}\label{eq:BR}
 2D(\mathbf u)=\sum_{n\ge0}T_{\mathbf u}(n),
\end{equation}
with equality in $[0,+\infty]$.

\begin{corollary}[Quasi-Sturmian consequences]
\label{cor:quasi-pal}
Let $\mathbf u$ be a recurrent aperiodic quasi-Sturmian word whose
language is closed under reversal.  Then
\begin{equation}\label{eq:quasi-pal}
 \Pal_{\mathbf u}(n)+\Pal_{\mathbf u}(n+1)=3
\end{equation}
for every sufficiently large $n$.  Consequently,
$D(\mathbf u)<\infty$.
\end{corollary}

\begin{proof}
For large $n$, equation~\eqref{eq:quasi-sturmian} gives
$p_{\mathbf u}(n)\le3n/2+1$ and $\Delta p_{\mathbf u}(n)=1$.  
Theorem~\ref{thm:main} therefore
gives~\eqref{eq:quasi-pal}, and hence $T_{\mathbf u}(n)=0$ for all sufficiently
large $n$.  The nonnegative series in~\eqref{eq:BR} consequently has
only finitely many nonzero terms, so $D(\mathbf u)<\infty$.
\end{proof}

\subsection{Reflection complexity}

The reflection complexity, introduced by Allouche, Campbell, Li, Shallit, and Stipulanti~\cite{AlloucheEtAl}, is denoted here by $r_{\mathbf u}(n)$;
it is the number of length-$n$ factors modulo the equivalence
$w\sim\overline w$.  When the language is closed under reversal, orbit
counting gives
\begin{equation}\label{eq:reflection-orbits}
 r_{\mathbf u}(n)=\frac{p_{\mathbf u}(n)+\Pal_{\mathbf u}(n)}2.
\end{equation}
When a bi-infinite word $\mathbf x$ is viewed as a coloring of the $2$-regular tree, its radius-$j$ subword complexity in the sense of
\cite{KimLim,KimLeeLimSim} equals $r_{\mathbf x}(2j+1)$.

\begin{remark}[Recovered reflection recurrence]
Under the hypotheses of Corollary~\ref{cor:quasi-pal}, one obtains
\begin{equation}\label{eq:reflection-recurrence}
 r_{\mathbf u}(n+2)=r_{\mathbf u}(n)+1
\end{equation}
for every sufficiently large $n$.  
Indeed, applying
\eqref{eq:quasi-pal} at $n$ and $n+1$ gives
$\Pal_{\mathbf u}(n+2)=\Pal_{\mathbf u}(n)$.  
Also $p_{\mathbf u}(n+2)=p_{\mathbf u}(n)+2$ for large $n$.
Equation~\eqref{eq:reflection-orbits} now gives
\eqref{eq:reflection-recurrence}.

Dvo\v{r}\'akov\'a and Pelantov\'a~\cite[Theorem~21]{DvorakovaPelantova}
proved the stronger equivalence that
$r(n+2)=r(n)+1$ eventually if and only if the word is quasi-Sturmian and
the language of some tail of the infinite word is closed under reversal.
Thus
equation~\eqref{eq:reflection-recurrence} is a special case of their
characterization.
\end{remark}

\end{document}